\documentclass{amsart}

\usepackage{amsmath,amssymb,amsthm}
\usepackage{graphicx,a4wide}

\DeclareMathAlphabet{\mathbbold}{U}{bbold}{m}{n}

\DeclareSymbolFont{rsfscript}{OMS}{rsfs}{m}{n}
\DeclareSymbolFontAlphabet{\mathrsfs}{rsfscript}

\DeclareFontFamily{OMS}{rsfs}{\skewchar\font'177}
\DeclareFontShape{OMS}{rsfs}{m}{n}{%
      <5> rsfs5
      <6> <7> rsfs7
      <8> <9> <10> rsfs10
      <10.95> <12> <14.4> <17.28> <20.74> <24.88> rsfs10
      }{}
\def\calA{\mathrsfs{A}}

\def\calF{\mathrsfs{F}}

\def\calV{\mathrsfs{V}}

\DeclareMathOperator{\st}{st}

\DeclareMathOperator{\Alt}{Alt}

\def\sAlt{\widetilde{\mathop{\Alt^!}}}
\def\sAs{\mathop{\calA_{(2)}}}

\DeclareMathOperator{\Ind}{Ind}
\DeclareMathOperator{\sgn}{sgn}

\DeclareMathOperator{\gr}{gr}
\DeclareMathOperator{\qgr}{qgr}

\makeatletter

\theoremstyle{plain}

\newtheorem*{itheorem}{Main Result}
\newtheorem {theorem}{Theorem}

\newtheorem {corollary}{Corollary}

\newtheorem {proposition}{Proposition}

\theoremstyle{definition}

\newtheorem {definition}{Definition}
\newtheorem {remark}{Remark}

\begin{document}

\title{Dual alternative algebras\\ in characteristic three}
\author{Vladimir Dotsenko}
\address{School of Mathematics, Trinity College, Dublin 2, Ireland}
\email{vdots@maths.tcd.ie}

\begin{abstract}
We prove that the dimension of the arity~$n$ component of the operad of dual alternative algebras over a field of characteristic three is equal to~$2^n-n$, and describe the structure of the 
corresponding $S_n$-module. 
\end{abstract}

\maketitle

\section{Introduction}

A binary operation $\circ$ on a vector space $V$ is said to define an alternative algebra structure if its associator $(a_1,a_2,a_3)=(a_1\circ a_2)\circ a_3-a_1\circ(a_2\circ a_3)$ is an alternating function, that is 
 $$
(a_1,a_2,a_3)=(-1)^\sigma (a_{\sigma(1)},a_{\sigma(2)},a_{\sigma(3)}) \text{ for all } \sigma\in S_3.
 $$  
The operad $\Alt$ of alternative algebras~\cite{ZSSS} is the only known ``nice'' \cite{TypeAlg} quadratic operad which is not Koszul; this fact has been recently established in characteristic zero by Dzhumadil'daev and Zusmanovich~\cite{DZ}; they computed dimensions of the first few components of the operad $\Alt$ and its Koszul dual~$\Alt^!$, and used the power series criterion due to Ginzburg and Kapranov~\cite{GK} to show that these operads fail to be Koszul. The operad $\Alt^!$ has a very nice description itself; it controls ``dual alternative algebras'', that is associative algebras satisfying the identity
\begin{equation}\label{altdualdef}
a_1a_2a_3+a_1a_3a_2+a_2a_1a_3+a_2a_3a_1+a_3a_1a_2+a_3a_2a_1=0.
\end{equation}
For the characteristic of the ground field different from $2$ and $3$, the corresponding variety of associative algebras coincides with the variety of associative algebras with the identity $x^3=0$; in particular, the corresponding operad is nilpotent, $\Alt^!(6)=0$. Interestingly enough, in characteristic~$3$ the behaviour of this operad changes dramatically, and it has nonzero operations of every arity. Dzhumadil'daev and Zusmanovich conjectured, based on computer experiments and some results of Etingof, Kim and Ma \cite{EKM}, that over a field of characteristic~$3$
 $$
\dim\Alt^!(n)=2^n-n,
 $$ 
and suggested an outline of a possible proof, which they however did not complete. It turns out that in fact this statement follows from the results of Lopatin~\cite{Lop} who, utilising computer calculations, was able to describe free algebras with the identity $x^3=0$ for any ground field; the multilinear components of those free algebras are identified with the components of the operad in question. The goal of this short note is to give a computer-independent proof of a stronger version of this result, also taking into account the action of the symmetric group $S_n$ on~$\Alt^!(n)$ by permutations of arguments. 
\begin{itheorem}
Over a field of characteristic three, the $S_n$-module $\Alt^!(n)$ admits a Specht filtration with the factors being the Specht modules
$\mathbbold{1}_n$ and $\Ind_{S_k\times S_{n-k}}^{S_n}(\mathbbold{1}_k\otimes\sgn_{n-k})$, $0\le k\le n-2$, each with multiplicity one.
\end{itheorem}
It turns out that the easiest way to prove this result is via the splitting of the associativity relation due to Livernet and Loday~\cite{MR}. Our strategy is to introduce a filtration on the operad $\Alt^!$ for which the obstructions to the associativity of the product $\frac12(ab+ba)$ disappear after considering the corresponding graded operad. In fact, for the representation-theoretic statement this is the key trick we use, as it turns out that the components of the graded operad are direct sums of Specht modules. However, there is another important idea without which no clear proof seems to be available: to compute dimensions of our operads, we treat them as shuffle operads, and examine their Gr\"obner bases~\cite{DK}. More precisely, we first modify the defining relations of $\Alt^!$ in such a way that they remain the same in characteristic three, define a ``smaller'' operad otherwise, but are not as restrictive as~\eqref{altdualdef}, and then ``degenerate'' those relations to obtain a much simpler operad that exhibits the same behaviour as the characteristic-three dual alternative operad in all characteristics different from~$2$. 

\subsection*{Acknowledgements. } Thanks are due to Ivan Yudin whose questions led us to discovering a gap in the original proof of the main theorem, and made us present the argument in a hopefully more conceptual way, and to the referee who pointed out that the conjecture of Dzhumadil'daev and Zusmanovich was actually proved in~\cite{Lop} before being formulated in~\cite{DZ}. Most of this project was completed while the author was supported by Grant GeoAlgPhys 2011-2013 awarded by the University of Luxembourg.

\section{Recollections}

For details on operads we refer the reader to the book~\cite{LV}, for details on Gr\"obner bases for operads~--- to the paper~\cite{DK}. Here we 
only recall the key notions used throughout the paper. 

By an operad we mean a monoid in one of the two monoidal categories: the category of symmetric collections equipped 
with the composition product or the category of nonsymmetric collections equipped with the shuffle composition product. The former kind of monoids is referred to as symmetric 
operads, the latter~--- as shuffle operads. We always assume that our collections are reduced, that is, have no elements of arity~$0$. In general, most ``natural'' operads are symmetric, 
but for computational purposes it is useful to treat them as shuffle operads. 

A very useful technical tool for dealing with (shuffle) operads is given by Gr\"obner bases. More precisely, operads can be presented via generators and relations, that is as quotients of free operads $\calF(\calV)$, where $\calV$ is the space of generators. The free shuffle operad generated by a given nonsymmetric collection admits a basis of ``tree monomials'' which can be defined combinatorially; a shuffle composition of tree monomials is again a tree monomial. 

There exist several ways to introduce a total ordering of tree monomials in such a way that the operadic compositions are compatible with that total ordering. There is also a combinatorial definition of divisibility of tree monomials that agrees with the naive operadic definition: one tree monomial occurs as a subtree in another one if and only if the latter can be obtained from the former by operadic compositions. A Gr\"obner basis of an ideal of the free operad is a system of generators of this ideal for which the leading term of every element of the ideal is divisible by one of the leading terms of elements of our system. Such a system of generators allows to perform ``long division'' modulo the ideal, computing for every element its canonical representative in the quotient. There exists an algorithmic way to compute a Gr\"obner basis starting from any given system of generators (``Buchberger's algorithm for shuffle operads''); combinatorially, it requires to ``resolve'' overlaps of the leading terms of relations. Each common multiple of the leading terms of two relations gives a new element, called the S-polynomial, which becomes a candidate for being adjoined to the Gr\"obner basis. If this element can be reduced to zero using long division with respect to the existing relations, it does not contribute; otherwise, its normal form is adjoined and the computation proceeds in the same way.

\section{The dual alternative operad}

\begin{definition}
The operad of dual alternative algebras is generated by one binary operation $a,b\mapsto ab$ subject to the identities
\begin{gather}
(a_1a_2)a_3=a_1(a_2a_3),\\
(a_1a_2)a_3+(a_1a_3)a_2+(a_2a_1)a_3+(a_2a_3)a_1+(a_3a_1)a_2+(a_3a_2)a_1=0.\label{altdual}
\end{gather}
\end{definition}

An important observation made in~\cite{DZ} is that in characteristic three the left-hand side of the second identity is actually a Lie monomial.

\begin{proposition}
Over a field of characteristic three, \eqref{altdual} can be replaced by an equivalent identity
\begin{equation}\label{altdual3}
[[a_1,a_2],a_3]+[[a_1,a_3],a_2]=0
\end{equation}
for the commutator $[a_1,a_2]=a_1a_2-a_2a_1$.
\end{proposition}

\begin{proof}
Indeed, we have
\begin{multline*}
[[a_1,a_2],a_3]+[[a_1,a_3],a_2]=\\=(a_1a_2-a_2a_1)a_3-a_3(a_1a_2-a_2a_1)+(a_1a_3-a_3a_1)a_2-a_2(a_1a_3-a_3a_1)=\\=
a_1a_2a_3+a_1a_3a_2+a_3a_2a_1+a_2a_3a_1-2a_2a_1a_3-2a_3a_1a_2=\\=a_1a_2a_3+a_1a_3a_2+a_2a_1a_3+a_2a_3a_1+a_3a_1a_2+a_3a_2a_1=0.
\end{multline*}
\end{proof}

If the characteristic of the ground field is different from two, the associative product can be split into a symmetric binary operation $a_1,a_2\mapsto a_1\cdot a_2=\frac12(a_1a_2+a_2a_1)$ and a skew-symmetric one $a_1,a_2\mapsto [a_1,a_2]=\frac12(a_1a_2-a_2a_1)$; this leads to a definition of the associative operad as the operad generated by a skew-symmetric operation and a symmetric operation subject to the identities~\cite{MR}
\begin{gather}
[a_1\cdot a_2,a_3]=a_1\cdot[a_2,a_3]+[a_1,a_3]\cdot a_2,\label{leibniz}\\
(a_1\cdot a_2)\cdot a_3-a_1\cdot(a_2\cdot a_3)=[a_2,[a_1,a_3]]. \label{def-ass}
\end{gather}
Note that the Jacobi identity $[a_1,[a_2,a_3]]+[a_2,[a_3,a_1]]+[a_3,[a_1,a_2]]=0$ for the bracket follows from the last of these identities. For a ground field of characteristic different from~$3$, in the operad whose defining relations are \eqref{leibniz}, \eqref{def-ass} and, in addition, \eqref{altdual3}, the relation $[[a_1,a_2],a_3]=0$ is satisfied. (Indeed, \eqref{altdual3} can be rewritten as $[[a_1,a_2],a_3]=[[a_3,a_1],a_2]$, which means that the nested commutator $[[a_1,a_2],a_3]$ is cyclically symmetric. This makes the Jacobi identity become $3[[a_1,a_2],a_3]=0$). The corresponding operad has been studied from various viewpoints previously, see e.g. \cite{BML,DL,EKM,FS,KR,Lat}. For pedagogical purposes, we first study this operad using the Gr\"obner bases technique. (Note that most of the references where this operad or the corresponding varieties of algebras, or the corresponding free algebras are discussed assume that the ground field is of characteristic zero; we do not need this assumption). Once that is done, we proceed with applying the same methods to a modification of that operad for which the same strategy applies, but which is better modelling the characteristic three behaviour of the dual alternative operad.

\section{A toy model}

\begin{definition}
The \emph{two-step Lie nilpotent associative operad} $\sAs$ is generated by a symmetric binary operation $a_1,a_2\mapsto a_1\cdot a_2$ and a skew-symmetric binary operation $a_1,a_2\mapsto [a_1,a_2]$ subject to relations
\begin{gather}
[[a_1,a_2],a_3]=0,\label{2step1}\\
[a_1\cdot a_2,a_3]=a_1\cdot[a_2,a_3]+[a_1,a_3]\cdot a_2,\label{2step2}\\
(a_1\cdot a_2)\cdot a_3-a_1\cdot(a_2\cdot a_3)=0.\label{2step3}
\end{gather}
\end{definition}
Our previous discussions show that this is the operad one obtains over a field of characteristic different from three after imposing the relations \eqref{leibniz}, \eqref{def-ass} and \eqref{altdual3}. Our goal is to explain how it can be approached from the operadic Gr\"obner bases viewpoint. The first step of the corresponding computation is discussed in~\cite[Ex.~8.10.11]{LV}. 

\begin{theorem}\label{sAs}
Over any field, $\dim\sAs(n)=2^{n-1}$.
\end{theorem}

\begin{proof}
We shall compute a Gr\"obner basis of defining relations of this operad; even though for the most convenient order of tree monomials the Gr\"obner basis we get is infinite, it is fairly regular-behaving, and is sufficient for computations. 

We introduce an order of tree monomials as follows: assuming that the bracket operation is greater than the product operation, we use a version of the path-lexicographic ordering of tree monomials~\cite{DK}, where we compare the path sequences of our trees using the lexicographic ordering of words. This ordering is compatible with the shuffle monoidal structure since any ordering of words gives a compatible ordering of tree monomials. For the quadratic relations, this ordering singles out the leading terms $[[a_1,a_2],a_3]$,$[[a_1,a_3],a_2]$, and $[a_1,[a_2,a_3]]$ of the relations \eqref{2step1}, the leading terms $[a_1\cdot a_2,a_3]$, $[a_1\cdot a_3,a_2]$ and $[a_1,a_2\cdot a_3]$ of the relations \eqref{2step2} (those that the usual rewriting in Poisson algebras would eliminate), and the leading terms $(a_1\cdot a_2)\cdot a_3$ and $(a_1\cdot a_3)\cdot a_2$ of the associativity relations~\eqref{2step3}. The arity $4$ part of the Gr\"obner basis consists of the elements $[a_1,a_4]\cdot[a_2,a_3]+[a_1,a_3]\cdot[a_2,a_4]$ and $[a_1,a_3]\cdot[a_2,a_4]+[a_1,a_2]\cdot[a_3,a_4]$, coming from the overlaps between the leading (and the only) term of the relations \eqref{2step1} and the leading terms of relations \eqref{2step2}. 

Moreover, let us introduce, for a set of Lie monomials (that is, tree monomials involving the bracket only) $c_1,\ldots,c_k$, their standardised product 
 $$
\st(c_1,\ldots,c_k):=(c_{i_1}\cdot(c_{i_2}\cdot(\cdots(c_{i_{k-1}}\cdot c_{i_k})\cdots))), 
 $$
where $c_{i_1}$, \ldots, $c_{i_k}$ is the re-ordering of the factors $c_1$, \ldots, $c_k$ according to their minimal arguments. Then for every arity $k\ge 4$ the Gr\"obner basis in arity $k$ consists of the elements
\begin{multline*}
\st([a_p,a_r],[a_q, a_s],a_1,\ldots,\hat{a_p},\ldots,\hat{a_q},\ldots,\hat{a_r},\ldots,\hat{a_s},\ldots)+\\+\st([a_p,a_q],[a_r, a_s],a_1,\ldots,\hat{a_p},\ldots,\hat{a_q},\ldots,\hat{a_r},\ldots,\hat{a_s},\ldots)
\end{multline*}
and
\begin{multline*}
\st([a_p,a_s],[a_q, a_r],a_1,\ldots,\hat{a_p},\ldots,\hat{a_q},\ldots,\hat{a_r},\ldots,\hat{a_s},\ldots)+\\+\st([a_p,a_q],[a_r, a_s],a_1,\ldots,\hat{a_p},\ldots,\hat{a_q},\ldots,\hat{a_r},\ldots,\hat{a_s},\ldots) 
\end{multline*}
for all $p<q<r<s$. This can be checked by a direct inspection.

Examining the leading terms of these elements, we conclude that as a vector space $\gr\sAs(n)$ has a basis consisting of the monomials
 $$
\st([a_{i_1},a_{i_2}],[a_{i_3},a_{i_4}],\ldots,[a_{i_{2s-1}},a_{i_{2s}}],a_{j_1},\ldots,a_{j_k}), 
 $$
where $n=2s+k$, $s\ge 0$, $i_1<\ldots<i_{2s}$, $j_1<\ldots<j_k$, and $$\{i_1,i_2,\ldots,i_{2s},j_1,\ldots,j_k\}=\{1,\ldots,n\}.$$ Therefore
 $$
\dim\sAs(n)=2^{n-1},
 $$
since we have $\binom{n}{2s}$ monomials for each $0\le 2s\le n$. 
\end{proof}

\begin{corollary}
Over any field, the the $S_n$-module $\sAs(n)$ is isomorphic to the direct sum of the Specht modules $$\Ind_{S_{2s}\times S_{n-2s}}^{S_n}(\sgn_{2s}\otimes\mathbbold{1}_{n-2s}), \quad 0\le 2s\le n,$$ each appearing with multiplicity one.
\end{corollary}

\begin{proof}
The basis elements constructed above are fixed by the subgroup $S_k\subset S_n$ permuting $\{j_1,\ldots,j_k\}$ and alternate under the action of  the subgroup $S_{2s}\subset S_n$ permuting the complement of~$\{j_1,\ldots,j_k\}$, so they generate a submodule of $\sAs(n)$ isomorphic to the Specht module $\Ind_{S_k\times S_{n-k}}^{S_n}(\mathbbold{1}_k\otimes\sgn_{n-k})$. Clearly, the sum of these submodules is direct.
\end{proof}

\begin{remark}
In fact, results of this section are valid over any commutative ring, since the leading terms of all elements we compute are equal to $\pm1$.
\end{remark}

\section{Proof of the main result}

\begin{definition}
The \emph{modified dual alternative operad} $\sAlt$ is generated by a symmetric binary operation $a_1,a_2\mapsto a_1\cdot a_2$ and a skew-symmetric binary operation $a_1,a_2\mapsto [a_1,a_2]$ subject to relations
\begin{gather*}
[[a_1,a_2],a_3]+[[a_1,a_3],a_2]=0,\\
[a_1\cdot a_2,a_3]=a_1\cdot[a_2,a_3]+[a_1,a_3]\cdot a_2,\\
(a_1\cdot a_2)\cdot a_3-a_1\cdot(a_2\cdot a_3)=[a_2,[a_1,a_3]].
\end{gather*} 
\end{definition}

Our previous discussions show that over a field of characteristic three this operad is isomorphic to the operad of dual alternative algebras, and over a field of characteristic different from three it is isomorphic to the operad~$\sAs$. To make it ``more interesting'' in characteristic different from three, we shall now introduce a further modification. Let us define a decreasing filtration by suboperads $F^p\sAlt$ of the operad $\sAlt$, letting $F^p\sAlt(n)$ be the span of all monomials of arity~$n$ built from the product $a,b\mapsto a\cdot b$ and the bracket $a,b\mapsto [a,b]$ for which the bracket is used at least $p$ times. Thus, $F^0\sAlt(n)=\sAlt(n)$, $F^1\sAlt(n)$ is the $n^\text{th}$ component of the operadic ideal generated by the bracket etc. Then because of ``associativity up to commutators''~\eqref{def-ass}, in the graded operad 
 $$
\gr\sAlt:=\bigoplus_{p}F^p\sAlt/F^{p+1}\sAlt
 $$
obtained from this filtration, the product $a,b\mapsto a\cdot b$ becomes associative, and the computations get simplified dramatically. In addition to this operad, we shall also study the operad $\qgr\sAlt$ whose defining relations are graded versions of the \emph{quadratic} relations of $\sAlt$, that is
\begin{gather*}
[[a_1,a_2],a_3]+[[a_1,a_3],a_2]=0,\\
[a_1\cdot a_2,a_3]=a_1\cdot[a_2,a_3]+[a_1,a_3]\cdot a_2,\\
(a_1\cdot a_2)\cdot a_3=a_1\cdot(a_2\cdot a_3).
\end{gather*}

\begin{theorem}
Over any field of characteristic different from two, $\dim\qgr\sAlt(n)=2^n-n$.
\end{theorem}

\begin{proof}
As above, we shall compute a Gr\"obner basis of defining relations of this operad.

We consider the same ordering of tree monomials as in Theorem~\ref{sAs}. For the quadratic relations, this ordering singles out the leading terms $[[a_1,a_2],a_3]$ and $[[a_1,a_3],a_2]$ of the relations \eqref{altdual3}, the leading terms $[a_1\cdot a_2,a_3]$, $[a_1\cdot a_3,a_2]$ and $[a_1,a_2\cdot a_3]$ of the relations \eqref{leibniz} (those that the usual rewriting in Poisson algebras would eliminate), and the leading terms $(a_1\cdot a_2)\cdot a_3$ and $(a_1\cdot a_3)\cdot a_2$ of the associativity relations. 

The arity $4$ part of the Gr\"obner basis consist of the elements $[a_1,[a_2,[a_3,a_4]]]$,  $[a_1,a_4]\cdot[a_2,a_3]+[a_1,a_3]\cdot[a_2,a_4]$, and $[a_1,a_3]\cdot[a_2,a_4]+[a_1,a_2]\cdot[a_3,a_4]$; here the first of these elements comes from the S-polynomial defined by the overlap of the leading term of the relation \eqref{altdual3} with itself (compare with the computation of the Gr\"obner basis for the mock-commutative operad~\cite{GeK} performed in~\cite{DK}), and both the second and the third element come from the overlaps between the leading term of~\eqref{altdual3} and the leading terms of relations \eqref{leibniz}. Note that the assumption on the characteristic different from two is crucial; in characteristic two our operad has a quadratic Gr\"obner basis.

Furthermore, similar to the case of Theorem~\ref{sAs}, for every arity $k\ge 4$ the Gr\"obner basis in arity $k$ consists of the elements
\begin{multline*}
\st([a_p,a_r],[a_q, a_s],a_1,\ldots,\hat{a_p},\ldots,\hat{a_q},\ldots,\hat{a_r},\ldots,\hat{a_s},\ldots)+\\+\st([a_p,a_q],[a_r, a_s],a_1,\ldots,\hat{a_p},\ldots,\hat{a_q},\ldots,\hat{a_r},\ldots,\hat{a_s},\ldots)
\end{multline*}
and
\begin{multline*}
\st([a_p,a_s],[a_q, a_r],a_1,\ldots,\hat{a_p},\ldots,\hat{a_q},\ldots,\hat{a_r},\ldots,\hat{a_s},\ldots)+\\+\st([a_p,a_q],[a_r, a_s],a_1,\ldots,\hat{a_p},\ldots,\hat{a_q},\ldots,\hat{a_r},\ldots,\hat{a_s},\ldots) 
\end{multline*}
for all $p<q<r<s$, and additionally the elements
\begin{multline*}
\st([a_p,a_r],[a_q,[a_s,a_t]],a_1,\ldots,\hat{a_p},\ldots,\hat{a_q},\ldots,\hat{a_r},\ldots,\hat{a_s},\ldots,\hat{a_t},\ldots)+\\+\st([a_p,a_q],[a_r, [a_s,a_t]],a_1,\ldots,\hat{a_p},\ldots,\hat{a_q},\ldots,\hat{a_r},\ldots,\hat{a_s},\ldots,\hat{a_t},\ldots)
\end{multline*}
\begin{multline*}
\st([a_p,a_s],[a_q,[a_r,a_t]],a_1,\ldots,\hat{a_p},\ldots,\hat{a_q},\ldots,\hat{a_r},\ldots,\hat{a_s},\ldots,\hat{a_t},\ldots)-\\-\st([a_p,a_q],[a_r, [a_s,a_t]],a_1,\ldots,\hat{a_p},\ldots,\hat{a_q},\ldots,\hat{a_r},\ldots,\hat{a_s},\ldots,\hat{a_t},\ldots)
\end{multline*}
\begin{multline*}
\st([a_p,a_t],[a_q,[a_r,a_s]],a_1,\ldots,\hat{a_p},\ldots,\hat{a_q},\ldots,\hat{a_r},\ldots,\hat{a_s},\ldots,\hat{a_t},\ldots)+\\+\st([a_p,a_q],[a_r, [a_s,a_t]],a_1,\ldots,\hat{a_p},\ldots,\hat{a_q},\ldots,\hat{a_r},\ldots,\hat{a_s},\ldots,\hat{a_t},\ldots)
\end{multline*}
\begin{multline*}
\st([a_q,a_r],[a_p,[a_s,a_t]],a_1,\ldots,\hat{a_p},\ldots,\hat{a_q},\ldots,\hat{a_r},\ldots,\hat{a_s},\ldots,\hat{a_t},\ldots)-\\-\st([a_p,a_q],[a_r, [a_s,a_t]],a_1,\ldots,\hat{a_p},\ldots,\hat{a_q},\ldots,\hat{a_r},\ldots,\hat{a_s},\ldots,\hat{a_t},\ldots)
\end{multline*}
\begin{multline*}
\st([a_q,a_s],[a_p,[a_r,a_t]],a_1,\ldots,\hat{a_p},\ldots,\hat{a_q},\ldots,\hat{a_r},\ldots,\hat{a_s},\ldots,\hat{a_t},\ldots)+\\+\st([a_p,a_q],[a_r, [a_s,a_t]],a_1,\ldots,\hat{a_p},\ldots,\hat{a_q},\ldots,\hat{a_r},\ldots,\hat{a_s},\ldots,\hat{a_t},\ldots)
\end{multline*}
\begin{multline*}
\st([a_q,a_t],[a_p,[a_r,a_s]],a_1,\ldots,\hat{a_p},\ldots,\hat{a_q},\ldots,\hat{a_r},\ldots,\hat{a_s},\ldots,\hat{a_t},\ldots)-\\-\st([a_p,a_q],[a_r, [a_s,a_t]],a_1,\ldots,\hat{a_p},\ldots,\hat{a_q},\ldots,\hat{a_r},\ldots,\hat{a_s},\ldots,\hat{a_t},\ldots)
\end{multline*}
\begin{multline*}
\st([a_r,a_s],[a_p,[a_q,a_t]],a_1,\ldots,\hat{a_p},\ldots,\hat{a_q},\ldots,\hat{a_r},\ldots,\hat{a_s},\ldots,\hat{a_t},\ldots)-\\-\st([a_p,a_q],[a_r, [a_s,a_t]],a_1,\ldots,\hat{a_p},\ldots,\hat{a_q},\ldots,\hat{a_r},\ldots,\hat{a_s},\ldots,\hat{a_t},\ldots)
\end{multline*}
\begin{multline*}
\st([a_r,a_t],[a_p,[a_q,a_s]],a_1,\ldots,\hat{a_p},\ldots,\hat{a_q},\ldots,\hat{a_r},\ldots,\hat{a_s},\ldots,\hat{a_t},\ldots)+\\+\st([a_p,a_q],[a_r, [a_s,a_t]],a_1,\ldots,\hat{a_p},\ldots,\hat{a_q},\ldots,\hat{a_r},\ldots,\hat{a_s},\ldots,\hat{a_t},\ldots)
\end{multline*}
\begin{multline*}
\st([a_s,a_t],[a_p,[a_q,a_r]],a_1,\ldots,\hat{a_p},\ldots,\hat{a_q},\ldots,\hat{a_r},\ldots,\hat{a_s},\ldots,\hat{a_t},\ldots)-\\-\st([a_p,a_q],[a_r, [a_s,a_t]],a_1,\ldots,\hat{a_p},\ldots,\hat{a_q},\ldots,\hat{a_r},\ldots,\hat{a_s},\ldots,\hat{a_t},\ldots)
\end{multline*}
for all $p<q<r<s<t$, and the elements
 $$
\st([a_p,[a_q,a_r]],[a_s\cdot [a_t,a_u]],a_1,\ldots,\hat{a_p},\ldots,\hat{a_q},\ldots,\hat{a_r},\ldots,\hat{a_s},\ldots,\hat{a_t},\ldots,\hat{a_u},\ldots)
 $$
for all $p<q<r<s<t<u$. This can be checked by a direct inspection.

Examining the leading terms of these elements, we conclude that as a vector space $\qgr\sAlt(n)$ has a basis consisting of the monomials
 $$
\st([a_{i_1},a_{i_2}],[a_{i_3},a_{i_4}],\ldots,[a_{i_{2s-1}},a_{i_{2s}}],a_{j_1},\ldots,a_{j_k}), 
 $$
where $n=2s+k$, $s\ge 0$, $i_1<\ldots<i_{2s}$, $j_1<\ldots<j_k$, and $\{i_1,\ldots,i_{2s},j_1,\ldots,j_k\}=\{1,\ldots,n\},$ together with the monomials
 $$
\st([a_{i_1},a_{i_2}],[a_{i_3},a_{i_4}],\ldots,[a_{i_{2s-1}},[a_{i_{2s}},a_{i_{2s+1}}]],a_{j_1},\ldots, a_{j_k}),
 $$
where $n=2s+1+k$, $s\ge 1$, $i_1<\ldots<i_{2s+1}$, $j_1<\ldots<j_k$, and 
 $\{i_1,\ldots,i_{2s},i_{2s+1},j_1,\ldots,j_k\}=\{1,\ldots,n\}.$
Therefore
 $$
\dim\qgr\sAlt(n)=\dim\gr\sAlt(n)=2^n-n,
 $$
since we have $\binom{n}{k}$ monomials for each $k=0,1,\ldots,n-2,n$ ($k=n-1$ does not occur). 
\end{proof}

\begin{corollary}\label{dirsum}
Over any field of characteristic different from two, the $S_n$-module $\qgr\sAlt(n)$ is isomorphic to the direct sum of Specht modules $$\Ind_{S_k\times S_{n-k}}^{S_n}(\mathbbold{1}_k\otimes\sgn_{n-k}), \quad 0\le k\le n-2 \text{ and } k=n,$$ each with multiplicity one.
\end{corollary}

\begin{proof}
The basis elements constructed above are fixed by the subgroup $S_k\subset S_n$ permuting $\{j_1,\ldots,j_k\}$ and alternate under the action of  the subgroup $S_{n-k}\subset S_n$ permuting the complement of~$\{j_1,\ldots,j_k\}$, so they generate a submodule of $\qgr\sAlt(n)$ isomorphic to the Specht module $\Ind_{S_k\times S_{n-k}}^{S_n}(\mathbbold{1}_k\otimes\sgn_{n-k})$. Clearly, the sum of these submodules is direct.
\end{proof}

\begin{remark}
In fact, the results above are valid over any commutative ring where $2$ is invertible, since the leading terms of all elements we compute are equal to $\pm1$ or $\pm2$.
\end{remark}

Now everything is ready to prove the main result of this paper.

\begin{theorem}
Over a field of characteristic three, the $S_n$-module $\Alt^!(n)$ has a Specht filtration with factors $$\Ind_{S_k\times S_{n-k}}^{S_n}(\mathbbold{1}_k\otimes\sgn_{n-k}), \quad 0\le k\le n-2 \text{ and } k=n,$$ each with multiplicity one.
\end{theorem}

\begin{proof}
First of all, we know that in characteristic three there is no difference between the $\mathbb{S}$-modules $\Alt^!$ and $\sAlt$. We shall now show that the operad $\gr\sAlt$ has quadratic relations, that is that $\gr\sAlt\simeq\qgr\sAlt$ (in general, $\gr\sAlt$ is a quotient of $\qgr\sAlt$). Corollary~\ref{dirsum} will then imply that $\Alt^!$ has a Specht filtration. 

The main idea is to compare the Gr\"obner bases of $\sAlt$ and of $\qgr\sAlt$, and observe that those Gr\"obner bases have the same sets of leading terms, so as nonsymmetric collections $\sAlt\simeq\qgr\sAlt$. Thus, the underlying graded vector spaces of $\sAlt\simeq\gr\sAlt$ and $\qgr\sAlt$ are the same, which shows that the $\mathbb{S}$-modules $\gr\sAlt$ and $\qgr\sAlt$ are isomorphic.

The key facts that allow to compare the two (infinite!) Gr\"obner bases are the following ones. First, a direct computation of resolutions of overlaps for ternary relations shows that incorporating lower terms in \eqref{def-ass} does not make the corresponding Gr\"obner basis larger. Furthermore, \eqref{def-ass} together with the degree~$4$ part of the Gr\"obner basis force the identity
 $$
a_1\cdot(a_2\cdot[a_3,a_4])=(a_1\cdot[a_3,a_4])\cdot a_2
 $$
in $\sAlt$, and this ``restricted associativity'' allows to reduce most of the computations of resolutions of overlaps to the corresponding computations in $\qgr\sAlt$, which instantly identifies the respective steps in the Gr\"obner bases computations for the two operads.
\end{proof}

\bibliographystyle{amsplain}

\providecommand{\bysame}{\leavevmode\hbox to3em{\hrulefill}\thinspace}
\providecommand{\MR}{\relax\ifhmode\unskip\space\fi MR }
\providecommand{\MRhref}[2]{%
  \href{http://www.ams.org/mathscinet-getitem?mr=#1}{#2}
}
\providecommand{\href}[2]{#2}

\end{document}